\documentclass[12pt,draft,a4paper]{amsart}
\usepackage[english]{babel}
\usepackage{amssymb}

\usepackage{graphicx}
\usepackage{geometry}
\usepackage{hyperref}
\hypersetup{colorlinks=true, linkcolor=blue, anchorcolor=blue, 
citecolor=red, filecolor=blue, menucolor=blue,
urlcolor=blue}
\numberwithin{equation}{section}
\newtheorem{theorem}{Theorem}[section]

\newtheorem{proposition}[theorem]{Proposition}
\newtheorem{lemma}[theorem]{Lemma}
\theoremstyle{remark}
\newtheorem{remark}{Remark}[section]

\theoremstyle{definition}
\newtheorem{definition}[theorem]{Definition}

\def\XXint#1#2#3{{\setbox0=\hbox{$#1{#2#3}{\int}$ }
\vcenter{\hbox{$#2#3$ }}\kern-.58\wd0}}

\begin{document}
\title
[Local smoothing for the massless Dirac-Coulomb equation]
{Local smoothing estimates for the massless \\Dirac-Coulomb equation in 2 and 3 dimensions}
\begin{abstract}
We prove local smoothing estimates for the massless Dirac equation with a Coulomb potential in $2$ and $3$ dimensions. Our strategy is inspired by \cite{burq1} and relies on partial wave subspaces decomposition and spectral analysis of the Dirac-Coulomb operator.
\end{abstract}
\date{\today}    
\author{Federico Cacciafesta}
\address{Federico Cacciafesta: 
SAPIENZA - Universit\`a di Roma,
Dipartimento di Matematica, 
Piazzale A.~Moro 2, I-00185 Roma, Italy}
\email{cacciafe@mat.uniroma1.it}
\urladdr{www.mat.uniroma1.it/people/~cacciafesta}

\author{\'Eric S\'er\'e}
\address{\'Eric S\'er\'e: CEREMADE (UMR CNRS 7534), Universit\'e Paris-Dauphine,
Place de Lattre de Tassigny,
F-75775 Paris Cedex 16}
\email{sere@ceremade.dauphine.fr}
\urladdr{www.ceremade.dauphine.fr/~sere}


\bibliographystyle{plain}

\maketitle

\section{Introduction and generalities}
The massless Dirac equation with an electric Coulomb potential reads
\begin{equation}\label{diraccoul}
\begin{cases}
\displaystyle
 iu_t+\mathcal{D}_nu-\frac{\nu}{|x|} u=0,\quad u(t,x):\mathbb{R}_t\times\mathbb{R}_x^n\rightarrow\mathbb{C}^{N}\\
u(0,x)=f(x)
\end{cases}
\end{equation}
where  the massless Dirac operator $\mathcal{D}_n$ is defined in terms of the Pauli matrices
\begin{equation}
\sigma_1=\left(\begin{array}{cc}0 & 1 \\1 & 0\end{array}\right),\quad
\sigma_2=\left(\begin{array}{cc}0 &-i \\i & 0\end{array}\right),\quad
\sigma_3=\left(\begin{array}{cc}1 & 0\\0 & -1\end{array}\right)
\end{equation}
as 
$$
\mathcal{D}_2=-i(\sigma_1\partial_{x}+\sigma_2\partial_{y})=
\left(\begin{array}{cc}0 & -i\partial_z \\-i\partial_{\overline{z}} & 0\end{array}\right)
$$
(we denote $\partial_z=\partial_x-i\partial_y$ and $\partial_{\overline{z}}=\partial_x+i\partial_y$) in dimension $n=2$ with $N=2$, and 
$$
\mathcal{D}_3=-i\displaystyle\sum_{k=1}^3\alpha_k\partial_k=-i(\alpha\cdot\nabla)
$$ 
where the $4\times 4$ Dirac matrices are given by
\begin{equation}
\alpha_k=\left(\begin{array}{cc}0 & \sigma_k \\\sigma_k & 0\end{array}\right),\quad k=1,2,3
\end{equation}
in dimension $n=3$, with $N=4$.

The Dirac equation is widely used in physics to describe relativistic particles of spin $1/2$ (see e.g. \cite{landlif2}, \cite{thaller}). In particular, the $2D$ massless equation is used as a model for charge carriers in graphene, a layer of carbon atoms arranged in a honeycomb lattice (see e.g. the physics survey \cite{neto} and the mathematical papers \cite{fef}, \cite{ha-le-sp}, \cite{ma-sie}). Note that there have been many rigorous results in the recent years on the stationary (massive) Dirac equation, essentially in dimension 3 (see e.g. \cite{lewser} for references). Comparatively, its dynamics has been less investigated up to now.

The $\sigma_j$ and $\alpha_j$ matrices were introduced in view of making the Dirac operator a square root of the Laplace operator: therefore, they satisfy by construction the following anticommutating relations
\begin{equation*}
\sigma_j\sigma_k+\sigma_k\sigma_j=2\delta_{ik}\mathbb{I}_2,\quad j,k=1,2;
\end{equation*}
\begin{equation*}
\alpha_j\alpha_k+\alpha_k\alpha_j=2\delta_{ik}\mathbb{I}_4,\quad j,k=1,2,3,
\end{equation*}
(we recall that the spectrum of $\mathcal{D}_n$ is the whole line $\mathbb{R}$). These conditions ensure that
$$
(i\partial_t-\mathcal{D}_n)(i\partial_t+\mathcal{D}_n)=(\Delta-\partial^2_{tt})I_{N},
$$
which strictly links the dynamics of the \emph{free} massless Dirac equation to a system of $N$ decoupled wave equations. Therefore, dynamical properties for the free Dirac equation can directly be derived from their wave counterparts: dispersive properties of the flow can be mainly encoded in the celebrated family of \emph{Strichartz estimates}, which are given by
\begin{equation}\label{freestr}
\| e^{it\mathcal{D}_n}f\|_{L^p_t\dot H^{\frac{1}{q}-\frac{1}{p}-\frac{1}{2}}_x}\lesssim
\| f\|_{L^2}
\end{equation}
where the exponents $(p,q)$ are wave admissible, i.e. satisfy
$$
\frac{2}{p}+\frac{d-1}{q}=\frac{d-1}2,\quad 2\leq p\leq\infty,\quad\frac{2(d-1)}{d-3}\geq q\geq 2
$$
and $e^{it\mathcal{D}_n}$ represents the propagator for the solutions to \eqref{diraccoul} with $V=0$ for $n=2,3$.
We stress the fact that the estimate corresponding to the couple $(p,q)=(2,\infty)$ in dimension $3$, the so called \emph{endpoint}, that is
\begin{equation}\label{end}
\| e^{it\mathcal{D}_3}f\|_{L^2_tL^\infty_x}\lesssim\| f\|_{\dot{H}^1}
\end{equation}
is known to fail as the corresponding one for the 3D wave flow. In \cite{machi} a refined version of this estimate involving angular spaces was given. As a consequence, the well-posedness of the celebrated cubic nonlinear Dirac equation was established, assuming initial data in $\dot{H}^1$ with slight additional regularity on the angular variable. The problem was afterwards completely solved in \cite{bejher}. See also \cite{bejher2} for the $2D$ case. 

Another family of {\it a priori} estimates encoding informations about dispersion is given by the so called \emph{local-smoothing type estimates}, which turn to be particularly useful to handle potential-type perturbations. Different versions of these estimates are available: to the best of our knowledge the most general one in this setting is
\begin{equation}\label{freemor}
\|w_\sigma^{-1/2}e^{it\mathcal{D}_3}f\|_{L^2_tL^2_x}\leq C\|f\|_{L^2}
\end{equation}
where $w_\sigma(x)=|x|(1+|\log|x||)^\sigma$, $\sigma>1$. Both \eqref{freestr} and \eqref{freemor} can be found in \cite{danfan} (see also \cite{esc-ve}, \cite{bouss}). We are not aware of any corresponding result for the $2D$ case; nevertheless, the strategy developed in \cite{danfan} can be suitably adapted to this context to obtain similar results.

In recent years a lot of effort has been spent in order to investigate what happens to dispersive flows when potential perturbations come into play. In \cite{danfan} both \eqref{freestr} and \eqref{freemor} were proved for the flow $e^{it\mathcal({D}_3+V)}$ under the assumption 
$$
|V(x)|\leq\frac{\delta}{w_\sigma(x)},\qquad \sigma>1,\qquad n=3,
$$
for some $\delta$ sufficiently small, which by the way is a sufficient condition to guarantee that the operator $\mathcal{D}_3+V$ is selfadjoint (see \cite{thaller}). Similar results are available for magnetic potentials as well (see e.g. \cite{boussdanfan}, \cite{caccia3}). The angular endpoint estimates proved in \cite{machi} have been generalized in \cite{cac1}, \cite{cac2} to small potential perturbations, by introducing some new mixed Strichartz-smoothing estimates.  At the same time, some effort has been spent in order to find examples of potentials such that the correspoding flows do \emph{not} disperse (in the sense discussed above). This problem has been tackled for the magnetic Dirac equation in \cite{arrizfan} in $3D$, essentially showing that for vector potentials of the form $A(x)\cong |x|^{-\delta} Mx$ with $\delta\in (1,2)$ most of the mass of the solution is localized around a non-dispersive function.
These arguments suggest that, heuristically, the degree of homogeneity of the operator works as a threshold for the validity of dispersive estimates: therefore, for the Dirac equation, the Coulomb potential is to be thought of as a \emph{critical} case (we stress the fact that the restriction to the massless case seems unavoidable now). In the very last years the problem of understanding dispersive estimates for scaling invariant potentials has been approached for other dynamics: in \cite{burq1} and \cite{burq2} the authors have proved indeed that Morawetz and Strichartz estimates can be recovered for both the Schr\"odinger and wave equations with inverse square potentials, and in \cite{fan} the $L^1\rightarrow L^\infty$ time decay estimate is proved in the electromagnetic case. 

The aim of this paper is to adapt the techniques of \cite{burq1} to the massless Dirac Coulomb equation, and thus to prove local smoothing estimates for solutions to \eqref{diraccoul}. The strategy of proof of \cite{burq1} is fairly straightforward: the basic idea is to use a spherical harmonics decomposition to reduce the equation to a radial one and then rely on the Hankel transform to diagonalize it. Several problems arise when trying to adapt these ideas to the Dirac setting. The first difficulty is given by the fact that the Dirac operator does not preserve radiality, and therefore one needs a more complex, but still classical, spherical harmonics decomposition related to the $SU_2$ group action (see \cite{thaller}, section 4.6). Another issue is that we cannot use the Hankel transform, and we have to build up a relativistic analogue. To do this, we need to study the continuous spectrum of the Dirac-Coulomb model, both in $2$ and $3$ dimensions. The corresponding generalized eigenfunctions are well known (see for instance \cite{landlif2} and \cite{nov}), but their rich structure will raise some additional technical difficulties. We stress the fact that, to the best of our knowledge, this is the first result concerning dispersive properties of Dirac-Coulomb, an operator which plays a fundamental role in relativistic quantum chemistry. We hope that our result will be useful in the study of the dynamics of nonlinear models involving, for instance, several relativistic electrons in a molecule.

Before stating our main result, we introduce some notations that will be used throughout the paper.
\\\\
\textbf{Notations.}

We shall use the standard notation $\dot H^s$ for the homogeneous Sobolev space with the norm $\| f\|_{\dot H^s}=\| |D|^sf\|_{L^2}$ where $|D|=(-\Delta)^{1/2}$, and $L^p_tL^q_x=L^p(\mathbb{R}_t; L^q(\mathbb{R}^n_x))$ for the mixed space-time Strichartz spaces.
%
We denote with $\Omega^s$ the operator
\begin{equation*}
(\Omega^s\phi)(x)=|x|^s\phi(x)
\end{equation*}
and with a little abuse of notation we will use the same symbol to indicate the operators which are pointwise equal for all times,
\begin{equation*}
(\Omega^s\psi)(t,x)=|x|^s\psi(t,x).
\end{equation*}
We introduce the following sets of indices 
\begin{equation}\label{B2}
B_2(k):=\{k:k=1/2+m,m\in\mathbb{Z}\}
\end{equation}
\begin{equation}\label{B3}
B_3(k):=\{k:k\in \mathbb{Z}\backslash\{0\}\}.
\end{equation}
As most of the forthcoming objects will be significantly different in $2D$ or $3D$, we will often use the apex $n$ to distinguish the dimensions; this will not have to be confused with standard powers, but interpretation will be clear from time to time.
For the definition of the angular spaces $\mathcal{H}_{k}^n$, $n=2,3$ we refer to forthcoming Proposition \ref{thaller} and subsequent Remark \ref{rknot}.
%
\\\\

We are now ready to state our main Theorem.

\begin{theorem}\label{thmor}
Let $n=2,3$, $\nu_n\in\left(-\frac{n-1}2,\frac{n-1}2\right)$, and let $k_n\in B_n(k)$. Let $u$ be a solution of \eqref{diraccoul}. Then, for any 
$$1/2<\alpha<\sqrt{k_n^2-\nu_n^2}+1/2$$
 and any $f\in L^2((0,\infty))\otimes \mathcal{H}_{\geq {k_n}}^n$ there exists a constant $C=C(\nu,\alpha,k)$ such that the following estimate holds
\begin{equation}\label{mor}
\displaystyle
\left\|\Omega^{-\alpha}\left|\mathcal{D}_n-\frac{\nu_n}{|x|}\right|^{1/2-\alpha} u\right\|_{L^2_tL^2_x}\leq C \|f\|_{L^2}.
\end{equation}
\end{theorem}

\begin{remark}
The dependence on the dimension of $\nu$, which will be neglected in the rest of the paper, has been added in the statement of the Theorem to allow us to remark the different ranges of admissible charges in dimensions $2$ and $3$. These ranges are in fact imposed by the condition on the Dirac Coulomb operator to be self-adjoint, which is indeed guaranteed in those intervals (see Remark \ref{rkself}). We stress also the fact that for $\nu$ approaching the endpoints, respectively $\pm1$ in $3D$ and $\pm1/2$ in $2D$, the range of admissible exponents for the lowest value of $k_n$ in estimate \eqref{mor} shrinks to zero. 
\end{remark}

\begin{remark}
It seems to be possible to generalize our proof of Theorem \ref{thmor} to any space dimension $n\geq 2$. As it will be clear, the key ingredients are given indeed by the construction of a suitable integral operator which transforms the problem into an ODE, and then by the careful analysis of some interaction integrals between generalized eigenstates of the Dirac-Coulomb equation that naturally appear. Therefore, the crucial step consists in finding the generalized eigenstates of the Dirac-Coulomb equation: this has been performed in \cite{gu-ma} in higher space dimension $n\geq4$ by generalizing the spherical wave decomposition (see Proposition \ref{thaller}), i.e. by carefully analyzing the symmetries of the $SO(n)$ group. Anyway, we prefer not to deal with this problem here both because technicalities would lead us too far and because of the scarce physical interest of it.
\end{remark}

\begin{remark}
The endpoint case $\alpha=1/2$ in  \eqref{mor} would permit to recover with a standard argument the full set of Strichartz estimates for the perturbed flow (see e.g. \cite{burq1}); even though we don't have a concrete counterexample, this estimate seems to fail. Forthcoming remark \ref{fail} gives more details on the issue.
\end{remark}

\begin{remark}\label{rkself}
In order to have a unitary flow for equation \eqref{diraccoul} one needs the operator $\mathcal{D}_n-\frac{\nu}{|x|}$ to be selfadjoint; many papers have been devoted to the study of this property  (see e.g. \cite{Schmincke-72}, \cite{Wust-73}, \cite{nenciu}, \cite{Klaus-Wust-78}, \cite{estlos}, \cite{estloss2}, \cite{arrveg}) in the $3D$ setting (note that the methods introduced in these papers are quite general and can be adapted to other dimensions). In $3D$ the essential self-adjointness is in fact guaranteed if $|\nu|\leq\frac{\sqrt{3}}2$, while in the interval $\frac{\sqrt3}2<|\nu| \leq 1$ it is still possible to build a distinguished selfadjoint extension. In the $2D$ case, as soon as $\nu$ is nonzero, the Dirac-Coulomb operator $\mathcal{D}_2-\frac{\nu}{|x|}$ defined on $C^\infty_0(\mathbb{R}^2)$ is not essentially selfadjoint (this is an immediate consequence of Theorem 4.16 in \cite{thaller}). But C. Warmt (\cite{wa}, Satz 2.2.6) recently proved the existence of a distinguished self-adjoint extension if $|\nu| < 1/2$. Therefore, in the range of $\nu_n$ for which we prove our Theorem \ref{thmor}, the operator $\mathcal{D}_n-\frac{\nu}{|x|}$ is self-adjoint (in the sense just discussed).
\end{remark}

The plan of the paper is the following: in section \ref{partial} we build the setup, reviewing the theory of partial wave subspaces and the spectrum of the Dirac-Coulomb operator in $2D$ and $3$D; then we define our analogue of the Hankel transform used in \cite{burq1}.
Section \ref{proof} is devoted to the proof of Theorem \ref{thmor}.
A short appendix containing some generalities on special functions  is included at the end for the reader's convenience.
\\\\
\textbf{Acknowledgments.} The first named author is supported by the FIRB 2012 "Dispersive dynamics, Fourier analysis and variational methods" funded by MIUR (Italy). The second named author acknowledges support from the ANR ``NoNAP'' of the french ministry of research.

\section{The setup.}\label{partial}

We devote this section to recall some classical facts about the radial Dirac operator and his spectrum, including a short discussion on how to explicitly calculate the generalized eigenstates of the continuous spectrum. For the sake of brevity we will omit several details, referring to \cite{landlif2} and \cite{thaller} for a complete discussion and for more informations on the topic. Also, we construct an integral transform that will be crucial in our proof.

\subsection{Partial wave decomposition and continuous spectrum}
The main ingredient we need to introduce is the so called \emph{partial wave decomposition}. We collect in Proposition \ref{thaller} the tools we will need in what follows that is, essentially, the fact that the Dirac-Coulomb operator can be seen as a radial operator with respect to some suitable decomposition, both  in dimensions $2$ and $3$. We will denote for brevity with $\mathcal{A}_2$ and $\mathcal{A}_3$ the following families of indices
\begin{equation}\label{j2}
\mathcal{A}_2:=\left\{k\in\mathbb{Z}+\frac12\right\};
\end{equation}
\begin{equation}\label{j3}
\mathcal{A}_3:=\left\{j=\frac12,\frac32,\dots, \: m_{j}=-j,-j+1,\dots,+j, \: k_{j}=\pm(j+1/2)\right\}
\end{equation}
Moreover, we define the $4$-vectors
\begin{equation}\label{Xi}
\Xi^+_{m_j,\mp(j+1/2)}=
\left(\begin{array}{cc}i \Omega^{m_j}_{j\mp1/2}\\
0\end{array}\right),\qquad
\Xi^-_{m_j,\mp(j+1/2)}=
\left(\begin{array}{cc}0\\ \Omega^{m_j}_{j\pm1/2},
\end{array}\right)
\end{equation}
where
\begin{equation*}
\Omega^{m_j}_{j-1/2}=\displaystyle\frac{1}{\sqrt{2j}}
\left(\begin{array}{cc}\sqrt{j+m_j}\:Y^{m_j-1/2}_{j-1/2}\\
\sqrt{j-m_j}\:Y^{m_j+1/2}_{j-1/2}
\end{array}\right)
\end{equation*}
\begin{equation*}
\Omega^{m_j}_{j+1/2}=\displaystyle\frac{1}{\sqrt{2j+2}}
\left(\begin{array}{cc}
\sqrt{j+1-m_j}\:Y_{j+1/2}^{m_j-1/2}\\
-\sqrt{j+1-m_j}\:Y_{j+1/2}^{m_j+1/2}
\end{array}\right)
\end{equation*}
and $Y^m_l$ are the usual spherical harmonics.

\begin{proposition}\label{thaller}
Let $n=2,3$. Then, we can define unitary isomorphisms between the Hilbert spaces 
$$L^2(\mathbb{R}^2)^{2} \cong L^2((0,\infty),dr)\bigotimes_{k\in\mathcal{A}_2} h^2_k,$$
$$L^2(\mathbb{R}^3)^{4} \cong L^2((0,\infty),dr)\bigotimes_{{j,m_j,k_j}\in\mathcal{A}_3} h^3_{j,m_j,k_j}$$
by the following decompositions
\begin{equation}\label{iso2}
\Phi_2(x)=\sum_{k\in\mathcal{A}_2}\frac{1}{2\sqrt{\pi}}\left(\begin{array}{cc} F_{k}(r)e^{i(k-1/2)\theta}\\
 G_{k}(r)
e^{i(k+1/2)\theta}
\end{array}\right),
\end{equation}
\begin{equation}\label{iso3}
\Phi_3(x)=\sum_{j,m_j,k_j\in\mathcal{A}_3}\frac{1}{2\sqrt{\pi}}\left(\begin{array}{cc} F_{j,m_j,k_j}(r)\Xi^+_{j,m_j,k_j}(\theta,\phi)\\
 G_{j,m_j,k_j}(r)
\Xi^-_{j,m_j,k_j}(\theta,\phi)
\end{array}\right)
\end{equation}
which hold for any $\Phi_n\in L^2(\mathbb{R}^n,\mathbb{C}^N)$.
 Moreover, the Dirac-Coulomb operator leaves invariant the partial wave subspaces $C^\infty_0((0,\infty))\otimes h^n_{k}$ and, with respect to the basis $\left\{e^{i(k-1/2)\theta},e^{i(k+1/2)\theta}\right\}$ and $\left\{\Xi^+_{j,m_j,k_j},\Xi^-_{j,m_j,k_j}\right\}$ is respectively represented by the radial matrices
\begin{equation}\label{raddir}
\displaystyle
d^2_{k}=\left(\begin{array}{cc}-\frac{\nu}{r}& -\frac{d}{dr}+\frac{k}{r}
\\  \frac{d}{dr}+\frac{k}{r} & -\frac{\nu}{r}\end{array}\right),\qquad
d^3_{j,m_j,k_j}=\left(\begin{array}{cc}-\frac{\nu}{r}& -\frac{d}{dr}+\frac{1+k_j}{r}
\\  \frac{d}{dr}-\frac{1-k_j}{r} & -\frac{\nu}{r}\end{array}\right).
\end{equation}

The Dirac-Coulomb opeator $\mathcal{D}_n-\frac{\nu}{r}$ on $C^\infty_0(\mathbb{R}^n)^N$ is unitary equivalent to the direct sum of "partial wave Dirac operators" $d_{j,m_j,k_j}$, 
\begin{equation*}
\mathcal{D}_2-\frac{\nu}{r}\cong \bigoplus_{k}d_{k}^2,\qquad
\mathcal{D}_3-\frac{\nu}{r}\cong \bigoplus_{j,m_k,k_k}d^3_{j,m_k,k_k}
\end{equation*}
For fixed $\overline{k}_2\in B(k_2)$ and $\overline{k}_3\in B(k_3)$ we denote with
\begin{equation}\label{HH}
\mathcal{H}_{\geq \overline{k}_2}^2=\bigoplus_{k:|k|\geq|\overline{k}_2|}h_{k},\qquad
\mathcal{H}_{\geq \overline{k}_3}^3=\bigoplus_{j:j\leq|k_3|-1/2,m_{j},k_{j}:}h_{j,m_{j},k_{j}}.
\end{equation}
\end{proposition}
\begin{proof}
See \cite{thaller}.
\end{proof}

\begin{remark}\label{rknot}
Notations in Proposition \ref{thaller} are very heavy; this is partly due to the richness of the Dirac operator itself and partly  to the fact that the spherical coordinates decomposition changes in dimensions $2$ and $3$. We should stress the fact that formulas \eqref{raddir} only depend on $k$: therefore, in what follows we will abuse notations neglecting the dependence on the other parameters in the case $n=3$ in order to have a unified presentation. 

\end{remark}

We now show how the functions of the continuous spectrum for the Dirac-Coulomb operator are obtained. We want to deal with the eigenvalue equation
\begin{equation}\label{eigenpro}
\left(\mathcal{D}_n-\frac{\nu}{|x|}\right)\Psi^n=\varepsilon\Psi^n,\qquad \varepsilon>0,\quad n=2,3.
\end{equation}
The application of Proposition \ref{thaller} suggests to seek for solutions in the form, for a fixed value of $\varepsilon$,
\begin{equation}\label{eigpsi2}
\Psi_{k,\varepsilon}^2(r,\hat{x})=
\left(\begin{array}{cc}F^2_{k,\varepsilon}(r)e^{i(k-1/2)\theta}\\\:G^2_{k,\varepsilon}(r)e^{i(k+1/2)\theta}
\end{array}\right),\qquad n=2
\end{equation}
\begin{equation}\label{eigpsi3}
\Psi_{k,\varepsilon}^3(r,\hat{x})=
\left(\begin{array}{cc}F^3_{k,\varepsilon}(r)\Xi^+_{k}(\hat{x})\\\:G^3_{k,\varepsilon}(r)\Xi^-_{k}(\hat{x})
\end{array}\right),\qquad n=3
\end{equation}
so that equation \eqref{eigenpro} will in fact reduce to a sytem of ODE for the radial components $F^n_k(r)$, $G^n_k(r)$ which is the same in both dimensions, apart for some constants, and can be expliticly solved (for the details on these calculations we refer to \cite{landlif2} section IV for the $3D$ and to \cite{nov} for the $2D$ case). We thus have the following formulas, for fixed values of $k$ and $\varepsilon>0$, and $n=2,3$,
\begin{equation}\label{eigen1}
F^n_{k,\varepsilon}=\frac{\sqrt{2}|\Gamma(\gamma+1+i\nu)|}{\Gamma(2\gamma+1)}e^{\pi\nu/2}(2\varepsilon r)^{\gamma-\frac{(n-1)}2}
 {\rm Re}
\left\{e^{i(\varepsilon r+\xi)}
{_1F_1}(\gamma-i\nu,2\gamma+1,-2i\varepsilon r)\right\}
\end{equation}
\begin{equation}\label{eigen2}
G^n_{k,\varepsilon}=\frac{i\sqrt{2}|\Gamma(\gamma+1+i\nu)|}{\Gamma(2\gamma+1)}e^{\pi\nu/2}(2\varepsilon r)^{\gamma-\frac{(n-1)}2}
 {\rm Im}
\left\{e^{i(\varepsilon r+\xi)}
{_1F_1}(\gamma-i\nu,2\gamma+1,-2i\varepsilon r)\right\}
\end{equation} 
where $_1F_1(a,b,z)$ are \emph{confluent hypergeometric functions} (see Appendix for the definition and some properties), $\gamma=\sqrt{k^2-\nu^2}$  and $e^{-2i\xi}=\frac{\gamma-i\nu}{k}$ is a phase shift. We will denote with
\begin{equation}\label{radeigpsi}
\psi_{k,\varepsilon}^n(r)=
\left(\begin{array}{cc}F^n_{k,\varepsilon}(r)\\\:G^n_{k,\varepsilon}(r)
\end{array}\right),\qquad n=2,3
\end{equation}
the vector of radial coordinates of the generalized eigenfunctions.

Notice that with this choice the following normalization condition holds for \eqref{eigpsi2}-\eqref{eigpsi3}
\begin{equation}\label{norm}
\int_0^{+\infty}\overline{\psi^n}_\epsilon(r)\psi^n_{\epsilon'}(r)r^{n-1}dr=\frac{\delta(\epsilon-\epsilon')}{\epsilon^{n-1}},\qquad n=2,3
\end{equation}
meaning with that that the function
$$
K(r,r')=r^{n-1}\int e^{-l\varepsilon}\overline{\psi^n}_\epsilon(r)\psi^n_{\epsilon'}(r)\varepsilon^{n-1}d\varepsilon
$$
is a $\delta$-approximation for $l\rightarrow0$.

\begin{remark}
We point out that, due to our choice of normalization \eqref{norm}, the homogeneity of the functions $\psi^n(r)$ with respect to $\varepsilon$ and $r$ is the same; this fact, that is not true anymore in the massive case, will be crucial in what follows, and therefore prevents the application of our strategy to the massive case.
\end{remark}

\begin{remark}
The sign of the Coulomb potential plays in fact a virtually unessential role in what is above, and the wave functions for repulsive fields can be obtained simply by changing the sign of $\nu$ in \eqref{eigen1}-\eqref{eigen2}.
\end{remark}

\begin{remark}\label{negative}
Formulas \eqref{eigen1}-\eqref{eigen2} give the "eigenstates" of the continuous spectrum corresponding to positive energies $\varepsilon>0$. Using a charge conjugation argument (or directly working on the explicit radial equations), it is possible to obtain the corresponding formulas for the radial components \eqref{eigpsi2} for negative energies: it turns indeed out that the function $\psi_{-k,\varepsilon}(r)$ solves
\begin{equation}\label{negativeig}
\tilde{d}_k\psi_{-k,\varepsilon}(r)=-\varepsilon \psi_{-k,\varepsilon},\qquad n=2,3,
\end{equation}
where the operator $\tilde{d}_k$ is the one defined in \eqref{raddir} with opposite sign of the charge $\nu$. This shows that the "eigenstates" corresponding to negative energies for equation \eqref{eigenpro} can be understood by mean of \eqref{negativeig} to be "eigenstates" corresponding to positive energies, with opposite value of $j$ and opposite charge. Therefore, denoting with $\tilde{F}$, $\tilde{G}$ the functions obtained by \eqref{eigen1}-\eqref{eigen2} by changing the sign of $\nu$, we obtain for the radial coordinates of the "eigenstates" corresponding to negative energies as in \eqref{radeigpsi} the following relation
$$
\psi^n_{k,-\varepsilon}(r)=
\left(\begin{array}{cc}F^n_{k,-\epsilon}(r)\\\:G^n_{k,-\epsilon}(r)\end{array}\right)=
\left(\begin{array}{cc}\tilde{F}^n_{-k,\epsilon}(r)\\\:\tilde{G}^n_{-k,\epsilon}(r)
\end{array}\right).
$$

\end{remark}

\subsection{The integral transform}
We now introduce the crucial integral transform that will be used in the proof of the main result, consisting in a projection on the continuous spectrum of the Dirac-Coulomb operator. Throughout this subsection we will omit the dependance on the dimensions of the functions, as the calculations carried out will be the same.
\begin{definition}\label{hanktras}
Let $\Phi\in L^2((0,\infty),dr)\otimes h_{k}$ for some fixed $k$ and let $\varphi(r)=(\varphi_1(r),\varphi_2(r))$ be the vector of its radial coordinates in decomposition \eqref{iso2}-\eqref{iso3}.  We define the following integral transform
\begin{equation}\label{H}
\mathcal{P}_k\varphi(\epsilon)=
\left(\begin{array}{cc}\mathcal{P}^+_k\varphi(\epsilon)\\
\mathcal{P}^-_k\varphi_k(\epsilon)
\end{array}\right)=
\left(\begin{array}{cc}\int_0^{+\infty}\psi_{k,\varepsilon}(r)\varphi(r)r^{n-1}dr\\
\mathcal{C}\left(\int_0^{+\infty}\psi_{k,-\varepsilon}(r)\varphi(r)r^{n-1}dr\right)
\end{array}\right)
\end{equation}
\begin{equation*}
=\int_0^{+\infty}H_{k}(\varepsilon r)\cdot\varphi(r)r^{n-1}dr
\end{equation*}
where we have introduced the matrix (keep in mind Remark \ref{negative})
\begin{equation}\label{matra}
H_{k}=\left(\begin{array}{cc}F_{k,\varepsilon}(r)\;G_{k,\varepsilon}(r)\\
F_{k,-\varepsilon}(r)\;G_{k,-\varepsilon}(r)
\end{array}\right)
\end{equation}
\end{definition}

We collect in the following proposition some important properties of the operator $\mathcal{P}_k$.
\begin{proposition}\label{properties}
For $n=2,3$ and any $\Phi\in L^2((0,\infty),dr)\otimes h_{k}$ the following properties hold:
\begin{enumerate}
\item
$\mathcal{P}_k$ is an $L^2$-isometry.
\item
$
\mathcal{P}_kd_k=\sigma_3\Omega\mathcal{P}_k.
$
\item
The inverse transform of $\mathcal{P}_k$ is given by
\begin{equation}\label{H-1}
\mathcal{P}_k^{-1}\varphi(r)=\int_0^{+\infty}H_{k}^{*}(\epsilon r)\cdot\varphi(\epsilon)\epsilon^{n-1}d\epsilon
\end{equation}
where $H_{k}^{*}=\left(\begin{array}{cc}F_{k,\varepsilon}(r)\;G_{k,-\varepsilon}(r)\\
G_{k,\varepsilon}(r)\;F_{k,-\varepsilon}(r)
\end{array}\right)
$.
\item
For every $\sigma\in\mathbb{R}$ we can define the fractional operators
\begin{equation}\label{fractiondef}
A_k^\sigma\varphi_k(r)=\mathcal{P}_k\sigma_3\Omega^\sigma\mathcal{P}_k^{-1}\varphi_k(r)=\int_0^{+\infty}S_k^\sigma(r,s)\cdot \varphi_k(s)s^{n-1}ds.
\end{equation}
where the integral kernel $S_k(r,s)$ is the $2\times2$ matrix given by
\begin{equation}\label{mattrix}
S_k(r,s)=\int_0^{+\infty}H_{k}(\epsilon r)\cdot H_{k}^{*}(\epsilon s)\epsilon^{n-1+\alpha}d\epsilon
\end{equation}
\end{enumerate}
\end{proposition}

\begin{remark}
Property $(1)$ allows, by standard arguments, to extend the definition of the operator $\mathcal{P}_k$ to functions in $L^2$.
\end{remark}

\begin{remark}\label{fractpow}
When summing on $k$, property \eqref{fractiondef} defines in a standard way fractional powers of the operator $|\mathcal{D}_n-\frac{\nu}{|x|}|$, which are used in the statement of Theorem \eqref{thmor}.
\end{remark}

\begin{proof}Property $(1)$ and $(2)$ come from the definition of $\mathcal{P}_k$, once noticed that
$$
\mathcal{P}^\pm_k(d_k\varphi)=\langle \psi_{k,\pm\epsilon},d_k\varphi\rangle=\langle d_k \psi_{k,\pm\epsilon},\varphi\rangle=\pm\epsilon\langle  \psi_{k,\pm\epsilon},\varphi\rangle
$$
and from normalization relation \eqref{norm}.
Property $(3)$ is a consequence of $(1)$ ($H^*_{k} $ is the adjoint of $H_{k}$).

To prove property $(4)$ we use the definition of $A_k$ to write
\begin{eqnarray*}
A_k^\sigma\varphi_k(r)&=&\mathcal{P}_k\Omega^\sigma\mathcal{P}_k^{-1}\varphi_k(r)
\\
&=&
\int_0^{+\infty}H_{k}(\epsilon r)\epsilon^{n-1+\sigma}\left(\int_0^{+\infty}H^*_{k}(\epsilon s)\varphi_k(s)s^{n-1}ds\right)d\epsilon.
\end{eqnarray*}
Exchanging the order of the integrals yields $(4)$.
%
\end{proof}

\section{Proof of Theorem \ref{thmor}}\label{proof}

We here prove our main result. Following \cite{burq1}, we rely on partial wave decomposition and on Proposition \ref{properties} to explicitly write the solutions on the single subspaces; then, a detailed and careful analysis of the kernel $S_k(r,s)$ of the operator $A_k$ defined in \eqref{fractiondef} will be performed to sum back.

\subsection{The strategy}
We start with decomposition \eqref{iso2}-\eqref{iso3} and work on a single spherical space which will be identified by $k\in\mathcal{A}_2$ as in \eqref{j2} and by the triple $\{j,m_k,k_j\}\in\mathcal{A}_3$ as in \eqref{j3} in the case $n=3$. With a slight abuse of notations, we will again neglect the dependence on everything but $k$ in the  $3D$ case (as noticed in Remark \ref{rknot}, the radial Dirac-Coulomb operator in dimension $3$ only depends on $k=k_j$, see \eqref{raddir}), as well as the dependence on the dimension for the various objects, in order to perform much readable and unified calculations. We thus rely on Proposition \ref{thaller} aiming to prove estimate \eqref{mor} on a fixed partial wave subspace with a constant $C$ bounded with respect to $k$, and eventually sum back. We thus take an initial condition $f$ with angular part in $h_k$ for some fixed $k$, and denote with $L_k f$ the solution to the initial value problem
\begin{equation}\label{diraccoulmod}
\begin{cases}
\displaystyle
 iu_t+d_ku=0,\\
u(0,x)=f(x)
\end{cases}
\end{equation}
with $d_k$ as in \eqref{raddir}.
We apply operator $\mathcal{P}_k$, which is an $L^2$-isometry, to the LHS side of estimate \eqref{mor} and use Proposition \ref{properties} to obtain (the application of the matrix $\sigma_3$ does not alter the $L^2$ norm)
\begin{equation*}
\|\mathcal{P}_k\Omega^{-\alpha}|d_k|^{1/2-\alpha}L_k f\|_{L^2_tL^2_x}=\|A_k^{-\alpha}\Omega^{1/2-\alpha}\mathcal{P}_kL_kf\|_{L^2_tL^2_x}
\end{equation*}
where we have used \eqref{fractiondef}. Now, $\mathcal{P}_kL_kf$ solves (see $(2)$ of Proposition \ref{thaller})
\begin{equation}\label{diraccoulbis}
\begin{cases}
\displaystyle
 i\partial_t\mathcal{P}_kL_kf+\sigma_3\Omega \mathcal{P}_kL_kf=0,\\
\mathcal{P}_kL_kf(0,\xi)=\mathcal{P}_kf(\xi),
\end{cases}
\end{equation}
so that the solution to this problem is explicitly given by
\begin{equation*}
\mathcal{P}_kL_kf(t,\xi)=e^{it\sigma_3\xi}\mathcal{P}_kf(\xi).
\end{equation*}
Fourier transforming in time (since $\mathcal{F}_{t\rightarrow\tau}$ is an $L^2$-isometry which commutes with $A_k^\sigma$) gives (with a minor abuse, we will neglect the $\sigma_3$ matrix from now on as we will only need to pick the $L^2$ norm and thus the sign will be unessential)
\begin{equation*}
(\mathcal{F}_t\mathcal{P}_kL_kf)(\tau,\xi)=(\mathcal{P}_kf)(\rho)\delta(\tau+\xi).
\end{equation*}

We can then write 
\begin{eqnarray*}
(A_k^{-\alpha}\Omega^{1/2-\alpha}\mathcal{F}_t\mathcal{P}_kL_kf)(\tau,\xi)&=&
\int_0^{+\infty}S_k^{-\alpha}(\xi,s)\delta(\tau+s)\mathcal{P}_kf(s)s^{\frac{2n-2\alpha-1}2}ds
\\
&=&
-S_k^{-\alpha}(\xi,\tau)\mathcal{P}_kf(\tau)\tau^{\frac{2n-2\alpha-1}2}.
\end{eqnarray*}
In view of proving \eqref{mor} we now need to calculate the $L^2$ norm in time and space of the quantity above which then gives (notice that, as the angular part in decomposition \eqref{iso2}-\eqref{iso3} is $L^2$-unitary, we only need to consider the radial integrals)
\begin{equation}\label{1}
\int_0^{+\infty}\int_0^{+\infty}((\mathcal{P}_kf)^*(\tau)S_k^{-\alpha}(\rho,\tau)^T)\cdot(S_k^{-\alpha}(\rho,\tau)(\mathcal{P}_kf)(\tau))\tau^{2n-2\alpha-1}\rho^2 d\rho.
\end{equation}
Since $S_k^p(\rho,\tau)^T=S_k^p(\tau,\rho)$, the integral in $d\rho$ yields $S_k^{-2\alpha}(\tau,\tau)$ and we are therefore left with
\begin{equation}\label{2}
\int_0^{+\infty}(\mathcal{P}_kf)^*(\tau)S_k^{-2\alpha}(\tau,\tau)(\mathcal{P}_kf)(\tau)\tau^{2n-2\alpha-1} \:d\tau.
\end{equation}

To conclude the proof of the estimate on the $k$-th space we need to bound \eqref{2} with the $L^2$-norm of $\mathcal{P}_kf$ (which we recall to be the $L^2$ norm of $f$). Using forthcoming Proposition \ref{tecint} we thus have
\begin{equation}\label{laststep}
\int_0^{+\infty}(\mathcal{P}_kf)^*(\tau)S_k^{-2\alpha}(\tau,\tau)(\mathcal{P}_kf)(\tau)\tau^{2n-2\alpha-1} \:d\tau
\end{equation}
\begin{eqnarray*}
&\leq& \int_0^{+\infty}{\rm Tr}(S_k^{-2\alpha}(\tau,\tau))|(\mathcal{P}_kf)(\tau)|^2\tau^{2n-2\alpha-1} \:d\tau
\\
&\leq& C_k \int_0^{+\infty}|(\mathcal{P}_kf)(\tau)|^2(\tau)\tau^{n-1}\:d\tau
\\
\nonumber
&=&C_k\|f\|_{L^2}.
\end{eqnarray*}
with $C_k$ a bounded constant of $k$. Expanding $f$ in partial waves and using the triangle inequality together with the orthogonality of spherical harmonics concludes the proof.

\subsection{The main integral}

The crucial step consists thus in evaluating
\begin{eqnarray}\label{trace}
\nonumber
{\rm Tr}(S^{-2\alpha}(\tau,\tau))&=&{\rm Re}\left[\int_0^{+\infty}F_{k,\varepsilon}(\tau)^2+
G_{k,\varepsilon}(\tau)^2\epsilon^{n-1-2\alpha} d\epsilon\right]
\\
&+&{\rm Re}\left[\int_0^{+\infty}F_{k,-\varepsilon}(\tau)^2+G_{k,-\varepsilon}(\tau)^2\epsilon^{n-1-2\alpha} d\epsilon\right]
\\
\nonumber
&=&{\rm Re}\big[I_1(\tau)+I_2(\tau)\big].
\end{eqnarray}
More than giving an explicit solution of such integrals, which seems to be possible but significantly complicated, we are more interested in proving the following
\begin{proposition}\label{tecint}
Let $n=2,3$ $k_n\in B(k)$ as defined in \eqref{B2}-\eqref{B3} and $\nu_n$ as in Theorem \ref{thmor}. Then for every $$1/2<\alpha<\sqrt{k_n^2-\nu_n^2}+1/2$$
there exists a constant $C_{k_n}$ such that 
$$
\rm{Tr}(S_k^{-2\alpha}(\tau,\tau))=C_{k_n} \tau^{-n+2\alpha}.
$$
Moreover,  $\displaystyle\sup_{k_n} |C_{k_n}|<\infty$
\end{proposition}

\begin{proof}
By recalling the structure of the generalized eigenstates \eqref{eigen1}-\eqref{eigen2}, we have that integral \eqref{trace} takes the form
\begin{equation}\label{thein}
I_1(\tau)=C_{k,\nu}\tau^{2\gamma-(n-1)}
\end{equation}
$$
\times\int_0^{+\infty}\epsilon^{2\gamma-2\alpha}e^{2i\epsilon\tau}{_1F_1}(\gamma- i\nu,2\gamma+1,-2i\epsilon \tau){_1F_1}(\gamma- i\nu,2\gamma+1,-2i\epsilon \tau)d\epsilon
$$
with $$C_{k,\nu}=\frac{2^{2\gamma+2}}{\pi}e^{\pi\nu}\frac{|\Gamma(\gamma+1+ i\nu|^2}{\Gamma(2\gamma+1)^2}e^{2i\xi}$$
where we recall that $\gamma=\sqrt{k^2-\nu^2}$. As the only difference between $I_1$ and $I_2$ is encoded in the sign of $\nu$ (see Remark \ref{negative}) which is virtually unessential, we will limit our discussion to $I_1$, the other case being completely analogous.

Integrals like \eqref{thein} have already been object of study in literature (see e.g. \cite{landlif}, \cite{gargoyle}, \cite{rey}) since they are connected with some different physical problems, as for instance matrix multipole elements for the Dirac Coulomb model, and they present some difficulties being not convergent at infinity as well as in the origin (see \eqref{ashyp}). Singularity at the origin can be dealt with by a standard integration by parts argument (as the one used to define $\Gamma$ function for negative real parts) and is discussed in \cite{sunwright} provided $2\gamma-2\alpha+1$ is not zero or a negative integer, while singularity at infinity is slightly more difficult to be handled.
To make things work, we need first of all to define $I_k(r,s)$ off-diagonal and then use a limiting procedure. Also, the introduction of a term of the form $e^{-\delta\epsilon}$ with $\delta\in\mathbb{R}^+$ in the integral is needed in order to make it convergent; we will then denote with $I^\delta_k$ the integral
$$
I^\delta_k(r,s)=C_{k,\nu}\tau^{2\gamma-(n-1)}
$$
$$
\times\int_0^{+\infty}\epsilon^{2\gamma-2\alpha}e^{\epsilon(i(r+s)-\delta)}{_1F_1}(\gamma- i\nu,2\gamma+1,-2i\epsilon r){_1F_1}(\gamma- i\nu,2\gamma+1,-2i\epsilon s)d\epsilon.
$$ 
This modified integral can be explicitly computed, and its solution is written in terms of a double Appel series, which unfortunately turns out to be not convergent for the values we are interested in. To deal with this problem, it is necessary to introduce an analytic continuation of such a series: we don't include the proofs of these results here as they can be found in corresponding references.

We begin with the following
\begin{lemma}\label{formint}
Let $\alpha$, $\beta$ such that $|\alpha|+|\beta|<|h|$ and $d$ not a negative integer nor zero. Then
\begin{equation}\label{intapp}
\int_0^{+\infty}t^{d-1}e^{-ht}F(a_1,b_1,\alpha t)F(a_2,b_2,\beta t)dt=h^{-d}\Gamma(d)F_2(d;a_1,a_2;b_1,b_2;\frac{\alpha}h,\frac{\beta}h)
\end{equation}
where $F_2$ is defined in \eqref{f2}.
\end{lemma}
\begin{proof}
See e.g. Lemma 1 in \cite{saadhall}.
\end{proof}

We should now point out that the ray of convergence of $F(\alpha;a_1,a_2;b_1;b_2;x,y)$ is $|x|+|y|<1$, which provides a first obstacle to the application of this result to our case. To overcome this difficulty, we rely on the following analytic continuation of $F_2$.
\begin{lemma}\label{rey}
If Re$(b_1-a_1)>0$, Re$(a_1)>0$
\begin{equation}\label{analcontin}
F_2(d;a_1,a_2,b_1,b_2;1,1)
\end{equation}
\begin{equation*}
=e^{-i\pi(b-a-d)}\frac{\Gamma(b_1)\Gamma(d-b_1+1)}{\Gamma(a_1)\Gamma(d-a_1+1)}{_3F_2}(d,b_2-a_2,d-b_1+1;b_2,d-a_1+1;1)
\end{equation*}
\begin{equation*}
+e^{i\pi a}\frac{\Gamma(b_1)\Gamma(d-b_1+1)}{\Gamma(b_1-a_1)\Gamma(d+a_1+1-b_1)}
{_3F_2}(d,a_2,d+1-b_1;b_2,d+a+1-b_1;1)
\end{equation*}
\end{lemma}

\begin{proof}
See \cite{rey}.

\end{proof}

For the sake of generality we have stated this Lemma with general values of the parameters; we list here for convenience the necessary substitutions to treat our case:
\begin{eqnarray}\label{values}
d&=&2\gamma+1-2\alpha;
\\
\nonumber
a_1=a_2&=&\gamma-i\nu;
\\
\nonumber
b_1=b_2&=&2\gamma+1.
\end{eqnarray}

Since necessary conditions of Lemma \ref{rey} are fullfilled with these values, this result allows us to pass to the limit $I_k^\delta(r,s)$ for $\delta\rightarrow 0$, $s\rightarrow r=:\tau$, and thus leaves us with
$$
I_1(\tau)=C_{k,\nu}\tau^{2\gamma-(n-1)}(2i\tau)^{-(2\gamma+1-2\alpha)}\Gamma(2\gamma+1-2\alpha)
$$
$$
\times F_2(2\gamma+1-2\alpha;\gamma-i\nu,\gamma-i\nu;2\gamma+1,2\gamma+1;1,1)
$$
with
$$C_{k,\nu}=\frac{2^{2\gamma+2}}{\pi}e^{\pi\nu}\frac{|\Gamma(\gamma+1+i\nu|^2}{\Gamma(2\gamma+1)^2}e^{2i\xi}.$$ Using now representation \eqref{analcontin} for $F_2$ and plugging in values \eqref{values}, we are led to 
\begin{equation}\label{Itau}
I_1(\tau)=\tau^{2\alpha-n}C_{k,\nu}(2i)^{-(2\gamma+1-2\alpha)}\Gamma(2\gamma+1-2\alpha)(A+B)
\end{equation}
where 
$$
A=e^{i\pi(\gamma-2\alpha-i\nu)}\frac{\Gamma(2\gamma+1)\Gamma(1-2\alpha)}{\Gamma(\gamma-i\nu)\Gamma(\gamma+2-2\alpha+i\nu)}
$$
$$
\times\sum_m\frac{(2\gamma+1-2\alpha)_m(\gamma+1+i\nu)_m(1-2\alpha)_m}{(2\gamma+1)_m(\gamma+2-2\alpha+i\nu)_mm!},
$$
$$
B=e^{i\pi(\gamma-i\nu)}\frac{\Gamma(2\gamma+1)\Gamma(1-2\alpha)}{\Gamma(\gamma+1+i\nu)\Gamma(\gamma+1-2\alpha-i\nu)}
$$
$$
\times\sum_m\frac{(2\gamma+1-2\alpha)_m(\gamma-i\nu)_m(1-2\alpha)_m}{(2\gamma+1)_m(\gamma+1-2\alpha-i\nu)_mm!}.
$$
Both the $_3F_2$ functions above turn to be convergent: we can indeed explicitly check convergence condition \eqref{convcond} to give in both cases $\alpha>0$. Notice also that assumption $\alpha<\gamma+1/2$ guarantees boundedness on the term $\Gamma(2\gamma+1-2\alpha)$. To conclude with, we now need to take the real part of \eqref{Itau} and show that it is a bounded function of $\gamma$. After some simplifications, we have
\begin{eqnarray}\label{reref}
\nonumber
{\rm Re}(I_1(\tau))&=&\frac{2^{1+2\alpha}e^{2\pi\nu}}\pi\frac{\Gamma(2\gamma+1-2\alpha)\Gamma(1-2\alpha)}{\Gamma(2\gamma+1)}
\\
\nonumber
&\times&{\rm Re}\Big[e^{-i\frac\pi2(2\alpha+1)}\frac{\Gamma(\gamma+1+i\nu)(\gamma-i\nu)}{\Gamma(\gamma+2-2\alpha-i\nu)}\sum_m\frac{(2\gamma+1-2\alpha)_m(\gamma+1+i\nu)_m(1-2\alpha)_m}{(2\gamma+1)_m(\gamma+2-2\alpha+i\nu)_mm!}
\\
\nonumber
&+&e^{i\frac\pi2(2\alpha-1)}\frac{\Gamma(\gamma+1-i\nu)}{\Gamma(\gamma+1-2\alpha-i\nu)}
\sum_m\frac{(2\gamma+1-2\alpha)_m(\gamma-i\nu)_m(1-2\alpha)_m}{(2\gamma+1)_m(\gamma+1-2\alpha-i\nu)_mm!}\Big]
\end{eqnarray}
where we have written $i^{-(2\gamma+1-2\alpha)}=e^{-i\frac\pi2(2\gamma+1-2\alpha)}$. Using now the basic property of the Pochhammer symbol $(x)_m=\Gamma(x+m)/\Gamma(x)$ many other terms simplify, leaving us with
$$
{\rm Re}(I_1(\tau))=\frac{2^{1+2\alpha}e^{2\pi\nu}}\pi
$$
$$
\times{\rm Re}
\Big[e^{-i\frac\pi2(2\alpha+1)}(\gamma-i\nu)\sum_m\frac{\Gamma(2\gamma+1-2\alpha+m)\Gamma(\gamma+1+m+i\nu)\Gamma(1-2\alpha+m)}{\Gamma(2\gamma+1+m)\Gamma(\gamma+2-2\alpha+m-i\nu)m!}
$$
$$
+e^{i\frac\pi2(2\alpha-1)}(\gamma-i\nu)
\sum_m\frac{\Gamma(2\gamma+1-2\alpha+m)\Gamma(\gamma+m-i\nu)\Gamma(1-2\alpha+m)}{\Gamma(2\gamma+1+m)\Gamma(\gamma+1-2\alpha+m+i\nu)m!}\Big].
$$
Stirling formula for Gamma function
$$
\Gamma(z)=\displaystyle\sqrt{\frac{2\pi}z}\left(\frac{z}e\right)^z\left(1+O\left(\frac1z\right)\right)
$$
shows that both the series above are asymptotic (in $\gamma$) to $\gamma^{-1}$, and therefore the proof is concluded.

\end{proof}

\begin{remark}\label{fail}
It should be noticed that when $\alpha\rightarrow 1$ the term $\Gamma(1-2\alpha)$ produces a singularity of order 1, as it is seen by applying well known property $\Gamma(z)\Gamma(1-z)=\pi/\sin(\pi z)$. Anyway, it can be checked that this singularity vanishes when taking the real part. Calculations are very easy in the unperturbed case (i.e. when $\nu=0$): in this case indeed everything is real except for the exponential terms, which become respectively $e^{-i\frac32\pi}$ and $e^{i\frac12\pi}$ which are purely imaginary. In the perturbed case things are a bit more complicated, but the problem has already been dealt with in literature by some limiting argument (see \cite{rey} and references therein). On the other hand, as $\alpha$ approaches $1/2$ the singularity produced by the term $\Gamma(1-2\alpha)$ is not balanced anymore (exponential terms have a nonzero real part): this rules out the case $\alpha=1/2$ and therefore the chance of recovering Strichartz estimates with standard arguments from our proof. The analysis of this problem will be the object of forthcoming works.
\end{remark}

\section{Appendix: special functions, properties and integrals.}

For reader's convenience we include a short appendix in which we recall definitions and basic properties of the special functions coming into play. We refer to \cite{slater} and \cite{erd} for further details and for a deeper insight of the argument. 
\begin{definition}
Given $a_1,\dots ,a_p$ and $b_1,\dots, b_q$ complex numbers, we define the \emph{generalized hypergeometric function} as
\begin{equation}\label{hyp}
_pF_q(a_1,\dots ,a_p;b_1,\dots, b_q;z):=\displaystyle\sum_{n=0}^\infty
\frac{(a_1)_n\dots(a_p)_n}{(b_1)_n\dots(b_q)_n}\frac{z^n}{n!}
\end{equation}
provided the sum of the series is finite, where we are using the \emph{Pochhammer symbols}
\begin{equation*}
(a)_0=1,\qquad (a)_n=a(a+1)\dots(a+n-1),\quad n\geq1.
\end{equation*}
\end{definition}
The first thing to investigate is, of course, under which conditions these series converge. As a very general framework, we can summarize the situation as follows
\begin{itemize}
\item if $p\leq q$ the series is convergent for all values of $z$;
\item if $p=q+1$ the series is convergent for $|z|<1$, and for the special values $z=1$ if 
\begin{equation}\label{convcond}{\rm Re}\left(\sum_{n=1}^qb_n-\sum_{m=1}^pa_m\right)>0\end{equation}
and $z=-1$ if
$${\rm Re}\left(\sum_{n=1}^qb_n-\sum_{m=1}^pa_m\right)>-1;$$
\item
if $p> q+1$ the series never converges except for $z=0$.
\end{itemize}

In this work we are especially interested in the \emph{confluent hypergeometric function} $_1F_1(a,c,z)$, which is thus the case $p=q=1$. This series is therefore known to be convergent for every finite value of $z$ (if $c$ is not zero nor a negative integer). Moreover, if $a$ is a negative integer or zero, $_1F_1$ reduces to a polynomial of degree $|a|$. Confluent hypergeometric functions can be thought of as a limiting case of the celebrated \emph{Gauss hypergeometric serie}, that namely is case $p=2$, $q=1$ in \eqref{hyp}, via the relation
\begin{equation*}
\displaystyle
_1F_1(a,c;z)=\lim_{b\rightarrow\infty}\:_2F_1(a,b;c;z/b).
\end{equation*}
These functions satisfy the differential equation
\begin{equation*}
zu''+(c-z)u'-au=0.
\end{equation*}
 The asymptotic behaviour of $_1F_1$ can be derived by writing
\begin{equation}\label{ashyp}
_1F_1(a,b,z)\cong\frac{\Gamma(b)}{\Gamma(b-a)}(-z)^{-a}G(a,a-b+1,z)+\frac{\Gamma(b)}{\Gamma(a)}e^zz^{a-b}G(b-a,1-a,z)
\end{equation}
where $G$ has the asymptotic series representation
\begin{equation*}
G(a,b,z)=1+\frac{ab}{1!z}+\frac{a(a+1)b(b+1)}{2!z^2}+...
\end{equation*}
Among the properties that confluent hypergeometric functions satisfy we recall the following one
\begin{equation}\label{twist}
F(a,c,z)=e^zF(c-a,c,-z).
\end{equation}
Among the possible generalizations of hypergeometric function a very useful one is given by the so called \emph{Appell functions} (see \cite{appel}). Historically, the motivation for introducing series of this type is connected to investigating different kind of products of two different Gauss functions. In particular, we are here interested in the following, which is sometimes called Appell function of the second type,
\begin{equation}\label{f2}
F_2(d;a_1,a_2;b_1;b_2;x,y)\displaystyle=\sum_{m=0}^\infty\sum_{n=0}^\infty\frac{(d)_{m+n}(a_1)_m(a_2)_n}{(b_1)_m(b_2)_nm!n!}x^my^n.
\end{equation}
 It can be seen that convergence condition for $F_2$ is given by $|x|+|y|<1$.

\end{document}